\documentclass[12pt]{amsart}
\usepackage{verbatim}
\usepackage[active]{srcltx}
\usepackage{amsthm}
\newtheorem{thm}{Theorem}[section]
\newtheorem{prop}[thm]{Proposition}
\newtheorem{cor}[thm]{Corollary}
\newtheorem{lemma}[thm]{Lemma}

\def\Z{\mathbb{Z}}
\def\d{{\partial }}

\def\Uc{\mathcal{U}}
\def\Bc{\mathcal{B}}
\def\Ec{\mathcal{E}}

\newcommand\map[1]{\stackrel{#1}{\longrightarrow}}

\title{Moore's theorem}
\author{V. Timorin}

\begin{document}
\begin{abstract}
In this (mostly expository) paper, we review a proof of the following old theorem of
R.L. Moore:
for a closed equivalence relation on the 2-sphere such that
all equivalence classes are connected and non-separating, and not all points
are equivalent, the quotient space is homeomorphic to the 2-sphere.
The proof uses a general topological theory close to but simpler than
an original theory of Moore.
The exposition is organized so that to make applications
of Moore's theory (not only Moore's theorem) in complex dynamics easier,
although no dynamical applications are mentioned here.
\end{abstract}

\maketitle

\section{Introduction}

In this (mostly expository) paper, we review a proof of the following theorem
of R.L. Moore \cite{Moore}:

\begin{thm}
  \label{moore}
  Let $\sim$ be a closed equivalence relation on the 2-sphere $S^2$ such that all
  equivalence classes are connected and non-separating, and not all points
  are equivalent.
  Then the quotient space $S^2/\sim$ is homeomorphic to the 2-sphere.
\end{thm}

Here, a {\em closed} equivalence relation on $S^2$ is defined as an equivalence
relation on $S^2$ that is a closed subset of $S^2\times S^2$.
A set $A\subset S^2$ is called {\em non-separating} if the complement $S^2-A$
is connected.

In the original proof of Theorem \ref{moore} (a complete proof is spread
over several publications \cite{Moore_foundations,Moore_postulates,Moore_simple,
Moore_separate,Moore,Moore_prime}, and it uses some results of \cite{Jan,Mullikin}),
Moore used his axiomatic
description of the sphere \cite{Moore_foundations}: a system of axioms on a topological space $X$
that guarantee that $X$ is homeomorphic to the 2-sphere.
We follow (in a much simplified form) this original approach.
Our objective, however, is not to give a simpler proof of Theorem
\ref{moore} but rather to introduce a topological theory (close to that
of Moore) describing the 2-sphere.
In fact, the author keeps several other applications of the
same theory in mind, motivated mainly by rational dynamics.
They will be described in separate preprint(s) or publication(s).

Let $X$ be a Hausdorff topological space.
Recall that a {\em path} in $X$ is a continuous map $\alpha:[0,1]\to X$.
A path $\alpha$ is called {\em simple} if $\alpha(t)\ne\alpha(t')$ unless $t=t'$.
A {\em simple curve} in $X$ is defined as the image of some simple continuous path.
If $A=\alpha[0,1]$ for a simple path $\alpha$, then $\alpha(0)$ and $\alpha(1)$
are called the {\em endpoints} of the simple curve $A$; and for $0\le a<b\le 1$,
the simple curve $\alpha[a,b]$ is called a {\em segment} of the simple curve $A$.
A {\em simple closed curve} in $X$ is defined as the image of $S^1$ under a continuous
embedding $\gamma:S^1\to X$.
Clearly, the image under $\gamma$ of a closed arc of $S^1$ is
a simple curve --- we call it an {\em arc} of the closed simple curve $\gamma(S^1)$.

One system of axioms that characterizes the 2-sphere is the following
(remarkably, only two axioms are enough!):
\begin{enumerate}
  \item {\em Jordan domain axiom.}
  For every simple closed curve $C$, the complement $X-C$ consists of two
  connected components
  called the {\em Jordan domains} bounded by $C$.
  Moreover, the boundary of any Jordan domain bounded by $C$ is exactly $C$.
  \item {\em Basis axiom.}
  There is a countable basis of topology in $X$ consisting of Jordan domains.
\end{enumerate}

\begin{thm}
  \label{moore-theory}
  Suppose that a compact, connected, locally path connected,
  Hausdorff topological space
  $X$ satisfies the Jordan domain axiom and the Basis axiom.
  Then $X$ is homeomorphic to the 2-sphere.
\end{thm}

The statement of Theorem \ref{moore-theory} is very close to Moore's
axiomatic description of $S^2$.
This statement can be strengthened.
One of the strongest versions is due to R.H. Bing \cite{Bing}:
a compact connected locally connected metrizable space $X$ with more than one point is
homeomorphic to the 2-sphere, provided that no embedded $S^0$ separates $X$,
and all embedded $S^1$ separate $X$.
The assumption that no pair of points (=embedded $S^0$) separates $X$
in Bing's theorem replaces the earlier assumption that
no simple curve separates $X$ \cite{Zippin,vK}.
Although these stronger characterizations of the 2-sphere imply
Moore's Theorem \ref{moore} more directly, we use Theorem \ref{moore-theory}
for two reasons: first, it is simpler to prove; second, a version
of this characterization can be stated dealing with more restricted classes of curves
(see Section \ref{s_rel}).
For the same reason, we work with the rather strong Basis axiom, although
we could have replaced it with a much weaker assumption that $X$
is second countable and
there exists at least one simple closed curve in $X$ \cite{vK}.
Another proof of Theorem \ref{moore} can be found in the book of Kuratowski
\cite{Kuratowski}.

We will first prove Theorem \ref{moore-theory} (Sections 2--4), and then deduce Theorem \ref{moore}
from it (Sections 5--6).
Many of the intermediate results should be also attributed to Moore, some
of the results are even older.
In the proof of generalized Jordan curve theorem, I have used some more recent ideas,
see e.g. \cite{Newman}.
It took me a considerable effort to recover a complete argument
presented here, and I hope that my exposition would be helpful to other people.
Some of the arguments are new.
Section \ref{s_rel} is aimed to extend Theorem \ref{moore-theory} and
make it suitable for other topological applications (cf. e.g. \cite{Tim}).

I am grateful to A. Epstein for communicating some useful references, for
carefully reading a draft of this paper, and making many useful comments and
suggestions.

\section{The extension property}

Let $X$ be a topological space satisfying the assumptions of Theorem \ref{moore-theory}.
In this section, we will prove the following important property:

\begin{thm}[Extension property]
\label{ext-prty}
  Let $D$ be a Jordan domain in $X$ bounded by a simple closed curve $C$.
  Then, for every pair of different points $a$, $b\in C$ there exists a
  simple curve connecting $a$ to $b$ and lying entirely in $D$, except for
  the endpoints.
\end{thm}

Let us first note that any Jordan domain is an open set, as a connected component
of the complement to a compact set.
An open connected set in a locally path connected space is path connected.
Thus all Jordan domains are path connected.

\begin{prop}
\label{interior}
  Every Jordan domain coincides with the interior of its closure.
\end{prop}

\begin{proof}
  Let $D$ be a Jordan domain bounded by $C$.
  Since $D$ is an open set, we have $D\subseteq {\rm Interior}(\overline D)$.
  Suppose that there is a point $x$ in the interior of $\overline D$
  such that $x\not\in D$.
  Then we must have $x\in\d D$.
  By the Jordan domain axiom, $\d D=C$ coincides with
  the boundary of the other Jordan domain $X-\overline D$ bounded by $C$.
  Thus the point $x$ is also on the boundary of $X-\overline D$.
  A contradiction.
\end{proof}

\begin{prop}
  \label{arc-nosep}
  No proper arc of a simple closed curve separates $X$.
\end{prop}

\begin{proof}
  Assume the contrary: $A$ is a proper arc of a simple closed curve $C$,
  and $D_1$, $D_2$ are different components of $X-A$.
  Since the set $C-A$ is connected, it lies in some component of $X-A$;
  suppose e.g. that $C-A\subset D_2$.
  Then $D_1$ is disjoint from $C$, therefore, it lies in some Jordan
  domain $D'_1$ bounded by $C$.
  The boundary of $D_1$ is disjoint from $D'_1$, hence
  $D_1$ must coincide with $D'_1$.
  However, $\d D_1\subseteq A$, a contradiction with the Jordan
  domain axiom, which says, in particular, that $\d D_1=C$.
\end{proof}

\begin{prop}
\label{subdomain}
 Let $U$ be a Jordan domain, and $C$ a simple closed curve in $\overline U$.
 Then there exists a Jordan domain $V$ bounded by $C$ such that $V\subseteq U$.
\end{prop}

\begin{proof}
  The Jordan domain $X-\overline U$ is a connected set disjoint from $C$.
  Therefore, it lies in one of the two Jordan domains bounded by $C$.
  It follows that the other Jordan domain bounded by $C$ (call it $V$)
  is contained in $\overline U$.
  Since $V$ is open, it is also contained in the interior of $\overline U$,
  which coincides with $U$ by Proposition \ref{interior}.
\end{proof}

\begin{lemma}[Simplification of paths]
\label{simplify-path}
Let $\alpha:[0,1]\to X$ be a continuous path in
$X$ such that $\alpha(0)\ne \alpha(1)$.
Then there exists a simple path $\beta:[0, 1] \to X$ such that
$\beta(0) = \alpha(0)$, $\beta(1) = \alpha(1)$, and $\beta[0, 1] \subseteq \alpha[0, 1]$.
\end{lemma}

We will call $\beta$ a {\em simplification} of $\alpha$.
Of course, in general, a simplification is not unique.
This lemma works in greater generality: $X$ can be replaced with any
Hausdorff topological space.

\begin{proof}
We say that two subintervals of $[0,1]$ are {\em essentially disjoint}
if their intersection is the empty set or a common endpoint.
Let $I_1 = [a_1, b_1]$ be a longest interval in $[0, 1]$ such that
$\alpha(a_1) = \alpha(b_1)$
(it is possible that $a_1 = b_1$; this happens if $\alpha$ was already simple).
Define $I_n = [a_n, b_n]$ inductively as a longest interval in
$[0, 1]$ essentially disjoint from $I_1$, $\dots$, $I_{n-1}$
such that $\alpha(a_n) = \alpha(b_n)$.
If the intervals $I_n$ thus defined contain more than one point,
then they are actually pair-wise disjoint,
due to the maximality of length assumption.
There is a continuous map $\xi:[0,1]\to [0,1]$ whose non-trivial fibers are precisely
the intervals $I_n$.
Define the map $\beta:[0,1]\to X$ as follows: if $\xi^{-1}(s)$ is
a single point $t$, then we set $\beta(s)=\alpha(t)$.
If $\xi^{-1}(s)=[a_n,b_n]$, then we set $\beta(s)=\alpha(a_n)=\alpha(b_n)$.
Clearly, $\beta$ is a continuous path.
It remains to prove that $\beta$ is simple.
Suppose that $\beta(s)=\beta(s')$ for $s\ne s'$.
Let $I$ be the smallest interval containing both $\xi^{-1}(s)$ and $\xi^{-1}(s')$.
For every $n$, we have either $I_n\subset I$ or $I_n\cap I=\emptyset$
(i.e. an endpoint of $I$ cannot be in the interior of $I_n$).
Choose the maximal $n$, for which $I_n\cap I=\emptyset$
(set $n=0$ if all $I_n$ are in $I$).
We must have $I_{n+1}\subset I$, which contradicts the maximality of $I_{n+1}$.
\end{proof}

\begin{prop}[Approximate extension property]
  \label{appr-ext}
  Let $U$ be an open connected subset of $X$.
  For any two boundary points $a,b\in\d U$ and any connected
  neighborhoods $V_a\ni a$,
  $V_b\ni b$, there exists a simple path $\alpha:[0,1]\to\overline U$
  such that $\alpha(0)\in\d U\cap V_a$, $\alpha(1)\in\d U\cap V_b$,
  and $\alpha(0,1)\subset U$.
\end{prop}

\begin{proof}
  Since $U$ is open and connected, it is path connected.
  Therefore, there exists a continuous path (not necessarily simple) in $U$
  connecting some point $a'$ of $V_a\cap U$ to some point $b'$ of $V_b\cap U$.
  There is also a continuous path in $V_a$ connecting $a$ to $a'$,
  and a continuous path in $V_b$ connecting $b'$ to $b$.
  Let $\tilde\alpha:[0,1]\to X$ be the concatenation of these three paths.
  The path $\tilde\alpha$ connects points $a$ and $b$ in $X$ but
  does not necessarily lie in $U$.
  We can choose a parameterization, however, so that $\alpha(t)$
  lies in $V_a$ for $t$ close to $0$, in $V_b$ for $t$ close to 1,
  and in $U$ for all other values of $t$.
  Let $(t_0,t_1)$ be the maximal interval such that $\tilde\alpha(t_0,t_1)
  \subset U$.
  Reparameterize the restriction of $\tilde\alpha$ to $[t_0,t_1]$ by $[0,1]$,
  and simplify this path (by Lemma \ref{simplify-path}).
  The path $\alpha$ thus obtained satisfies the desired properties.
\end{proof}

\begin{prop}
  \label{div-domain}
 Consider a Jordan domain $D$ and a pair of different points $a$, $c$
 in $\d D$.
 Let $ac$ be a simple curve connecting $a$ to $c$ and lying entirely in $D$,
 except for the endpoints.
 Then $ac$ divides $D$ into exactly two components; these
 components are Jordan domains.
\end{prop}

\begin{proof}
 Choose points $b$ and $d$ in such a way that the four points $a$, $b$,
 $c$, $d$ lie on the simple closed curve $\d D$ in this cyclic order.
 Define the simple closed curves $abc$ and $adc$ as the unions of $ac$
 with the arcs of $\d D$ bounded by $\{a,c\}$ and containing $b$ and $d$,
 respectively.
 Let $D_{abc}$ be the Jordan domain bounded by $abc$ and contained in $D$
 (it exists by Proposition \ref{subdomain}).
 Similarly, define $D_{adc}$ as the Jordan domain bounded by $adc$ and
 contained in $D$.
 Since $D_{abc}$ is contained in $D-ac$, but $\d D_{abc}$ is disjoint from
 $D-ac$, the domain $D_{abc}$ is a connected component of $D-ac$.
 Similarly, $D_{adc}$ is also a connected component of $D-ac$.
 We need to prove that there are no other connected components of $D-ac$.
 Assume the contrary: $\Omega$ is yet another connected component.

 The curve $ac$ is a proper arc of the simple closed curve $abc$,
 thus $ac$ cannot contain the boundary of $\Omega$ by Proposition
 \ref{arc-nosep}.
 It follows that there exists a point $e\in\d\Omega\cap\d D$
 different from $a$ and $c$.
 To fix the ideas, assume that $e\in abc$.
 On the other hand, we cannot have $\d\Omega\subseteq\d D$;
 otherwise $\Omega=D$.
 It follows that there exists a point $f\in\d\Omega\cap ac$
 different from $a$ and $c$.

 Choose sufficiently small disjoint Jordan neighborhoods $V_e\ni e$ and $V_f\ni f$.
 Then $V_e$ is disjoint from $ac$, and $V_f$ is disjoint from $\d D$.
 We also have that $V_e\cap\d\Omega\subset abc$ and
 $V_f\cap\d\Omega\subset ac$.
 By Proposition \ref{appr-ext}, there is a simple curve connecting a
 point in $V_e\cap abc$ to a point in $V_f\cap ac$ and lying entirely
 in $\Omega$, except for the endpoints.
 There is also a simple curve connecting a (possibly different) point in
 $abc-ac$ to a (possibly different) point in $ac$ and lying entirely
 in $D_{abc}$, except for the endpoints.
 The union of these two simple curves and two arcs of $abc$
 (one in $abc-ac$, the other in $ac$)
 is a simple closed curve $C$ in $\overline D$.
 Consider the Jordan domain $D_{ef}$ bounded by $C$ and contained in $D$.
 The boundary of $D_{ef}$ intersects both $\Omega$ and $D_{abc}$.
 Therefore, $D_{ef}$ itself intersects both $\Omega$ and $D_{abc}$
 (since these two sets are open).
 Thus $D_{ef}$ must also intersect $ac$.
 This is impossible, however, because in this case, the intersection of
 the boundary of $D_{ef}$ with $ac$ must have more than one connected component.
\end{proof}

\begin{prop}
  \label{crossing}
  Consider a Jordan domain $D$ and four different points $a$, $b$, $c$, $d$ in $\d D$
  that appear in this cyclic order.
  Let $ac$ (resp. $bd$) be a simple curve connecting $a$ with $c$
  (resp. $b$ with $d$) and lying entirely in $D$, except for the endpoints.
  Then the curves $ac$ and $bd$ intersect.
\end{prop}

\begin{proof}
  Let $D_{abc}$ and $D_{adc}$ be as in Proposition \ref{div-domain}.
  We know that $D$ is the union of $D_{abc}$, $D_{adc}$ and
  the curve $ac$ minus the endpoints.
  The points on the curve $bd$ that are close to $b$ must belong to
  $D_{abc}$ because $b$ is not in the closure of $D_{adc}$.
  Similarly, the points on $bd$ that are close to $d$ must belong to $D_{adc}$.
  It follows that $bd$ intersects $ac$.
\end{proof}

\begin{prop}\label{Jordan-intersect}
  Let $U$ and $V$ be Jordan domains, and $x\in \d U\cap V$.
  There exists a Jordan domain $W\subseteq U\cap V$ such that $x\in\d W$.
  Moreover, $W$ can be chosen as a connected component of $U\cap V$.
\end{prop}

\begin{proof}
  Let $u:[0,1]\to \d U$ be a surjective continuous path such that
  $u(0)=u(1)=x$, and $u(t)\ne u(t')$ for $t\ne t'$ unless
  $\{t,t'\}=\{0,1\}$.
  Define a {\em supporting interval} as a subinterval $[a,b]$ of $(0,1)$
  such that $u(a),u(b)\in\d V$, and there is a component of
  $\d V-\{u(a),u(b)\}$ that lies in $U$.
  We call this component a {\em penetrating arc}.
  Note that two different penetrating arcs are necessarily disjoint (because an endpoint
  of a penetrating arc cannot lie in another penetrating arc).
  Note that, by Proposition \ref{crossing}, any two supporting intervals are either
  nested or essentially disjoint.

  Since penetrating arcs are disjoint, they form a null-sequence
  (i.e. every subsequence contains a further subsequence converging to a point).
  It follows that the corresponding supporting intervals also form a null-sequence.
  In particular, in any subset of supporting intervals, there exists a
  longest interval with respect to the standard length (i.e. the length of $[a,b]$
  is $|b-a|$).

  Let $I_1$ be the longest supporting interval, $I_2$ the longest supporting
  interval essentially disjoint from $I_1$, etc.
  For any $n>1$, we define $I_n$ as the as the longest supporting interval
  essentially disjoint from $I_1$, $\dots$, $I_{n-1}$ (if any).
  In this way, we obtain a finite or countable sequence of essentially
  disjoint supporting intervals.
  In the case of countably many segments, their lengths tend to zero.

  Consider the simple closed curve $\Gamma$ obtained from $\d U$ by replacing
  every arc $u(I_k)=u[a_k,b_k]$ with a corresponding penetrating component of
  $\d V-\{u(a_k),u(b_k)\}$ (e.g. this component can be re-parameterized
  by the interval $[a_k,b_k]$).
  Clearly, $\Gamma$ contains $x$ and lies in $\overline U$.
  Therefore, there is a Jordan domain $W$ bounded by $\Gamma$ and contained in $U$.
  We claim that also $W\subseteq V$.
  Since $\Gamma$ intersects $V$ (e.g. $x\in\Gamma\cap V$),
  the domain $W$ also intersects $V$.
  It suffices to prove that $W$ is disjoint from the boundary of $V$.
  Assume the contrary: there is a component $A$ of $\d V\cap W$ connecting a pair of
  points in $\Gamma$.
  Since $W\subset U$, the arc $A$ belongs to some penetrating arc of $\d V$
  bounded by $u(a)$ and $u(b)$.
  The endpoints of $A$ are in $\Gamma$, therefore, $A$ coincides with this
  penetrating arc.
  By definition, $[a,b]$ is a supporting interval of $(0,1)$.
  By construction, it must belong to some $I_k=[a_k,b_k]$.
  Then the points $u(a)$ and $u(b)$ can only be in $\Gamma$ if $a=a_k$ and $b=b_k$.
  However, in this case, $u(a)$ and $u(b)$ cannot be endpoints of a component of
  $\d V\cap W$.
  The contradiction shows that $W\subseteq V$ and, therefore, $W\subseteq U\cap V$.
  Since all boundary points of $W$ belong to $\d U\cup\d V$, there is no
  connected open set contained in $U\cap V$ and having $W$ as a proper subset.
  It follows that $W$ is a connected component of $U\cap V$.
\end{proof}

\begin{prop}
\label{Un}
 Consider a Jordan domain $U$ in $X$ and a point $x$ on the boundary of $U$.
 Then there is a sequence of Jordan domains $U_n$ such that
 $$
 U_{n+1} \subset U_n \subset U,\quad\bigcap\overline U_n=\{x\}.
 $$
\end{prop}

\begin{proof}
 Let $V_n$ be a basis of neighborhoods of $x$ consisting of Jordan domains.
 Set $U_1$ to be a component of $U\cap V_1$, whose boundary is a simple
 closed curve containing $x$.
 The existence of such component follows from Proposition \ref{Jordan-intersect}.
 Set $U_2$ to be a component of $U_1\cap V_2$, whose boundary is a simple
 closed curve containing $x$, etc.
 In this way, we construct a sequence $U_n$ of Jordan domains with the desired
 properties.
\end{proof}

\begin{prop}
\label{connect-to-bdry}
Let $U$ be a Jordan domain in $X$.
For any point $o\in U$ and any point $x$ on the boundary of $U$,
there is a simple path $\beta:[0, 1] \to\overline U$ such that
$\beta(0) = o$, $\beta(1) = x$,
and $\beta(t) \in U$ for all $t\in [0, 1)$.
\end{prop}

\begin{proof}
Let $U_n$ be a sequence of Jordan domains such that
$U_{n+1}\subseteq U_n\subseteq U$ and $\bigcap\overline U_n=\{x\}$
(such sequence exists by Proposition \ref{Un}).
Take $x_n \in U_n$.
We can now define a continuous path $\alpha$ (not necessarily simple) as follows.
The restriction of $\alpha$ to $[0,1/2]$ is a path connecting $o$ to $x_1$ in $U$.
Next, the restriction of $\alpha$ to $[1/2,3/4]$ is defined as a path connecting $x_1$ to
$x_2$ in $U_1$.
Inductively, we define the restriction of $\alpha$ to $[1-2^{-n},1-2^{-n-1}]$
as a path connecting $x_n$ to $x_{n+1}$ in $U_n$.
We obtain a path $\alpha:[0, 1)\to U$ such that $\alpha(t)\to x$ as $t\to 1$.
(This is because $x$ is the only intersection point of $\overline U_n$).
Therefore, we can set $\alpha(1) = x$ thus obtaining a continuous path
$\alpha: [0; 1] \to\overline U$ such that $\alpha(t)\in U$ for all $t\in (0, 1)$.
Finally, we define $\beta$ as a simplification of $\alpha$.
\end{proof}

The Extension property (Theorem \ref{ext-prty}) follows from
Proposition \ref{connect-to-bdry} and Lemma \ref{simplify-path}.

\section{Simple intersections}

The boundaries of two Jordan domains can intersect in a complicated way.
In this section, we will develop a machinery that reduces many questions
on mutual position of Jordan domains to the case, where the boundaries
intersect in a simple way.
As in the preceding section, we consider a compact, connected,
locally path connected, Hausdorff topological space $X$ satisfying
the Jordan domain axiom and the Basis axiom.

We say that a simple closed curve $A$ in $X$ has only {\em simple intersections} with
a simple closed curve $B$ in $X$ if $A\cap B$ is a finite set,
and every two adjacent components of $A-B$ lie on different sides of $B$,
i.e. in different Jordan domains bounded by $B$.

\begin{lemma}
  Suppose that $A$ has only simple intersections with $B$.
  Then $B$ has only simple intersections with $A$.
\end{lemma}

\begin{proof}
  We need to prove that every two adjacent components of $B-A$ lie
  on different sides of $A$.
  Suppose not: two adjacent components $B_1$ and $B_2$ of $B-A$
  belong to the same Jordan domain $D$ bounded by $A$.
  Let $D'$ be the other Jordan domain bounded by $A$, and
  $b$ the common endpoint of $B_1$ and $B_2$.
  Denote by $A_1$ and $A_2$ the two components of $A-B$ incident to $b$.
  If $U$ is a sufficiently small Jordan domain neighborhood of $b$,
  then $U\cap A\subset A_1\cup A_2\cup\{b\}$.
  By Proposition \ref{Jordan-intersect}, there exists a Jordan
  domain $V$ such that $b\in\d V$, and $V$ is a connected component
  of $U\cap D'$.
  There exist two points $a_1\in\d V\cap A_1$ and $a_2\in\d V\cap A_2$
  and a component $C$ of $\d V\cap D'$ connecting $a_1$ with $a_2$.
  The curve $C$ is disjoint from $B$.
  Therefore, $a_1$ and $a_2$ must be in the same component of $X-B$; a contradiction.
\end{proof}

The main objective of this section is to prove the following

\begin{thm}[Covering property]
  \label{cov-prty}
  Let $\Uc$ be a finite open covering of $\overline D$, where $D$ is a Jordan domain in $X$.
  Then there exists a refinement $\Uc'$ of $\Uc$ of the same cardinality such that
  every element $U'\in\Uc'$ is a Jordan domain, whose boundary has only simple
  intersections with $\d D$.
\end{thm}

Recall that a covering $\Uc'$ is a {\em refinement} of $\Uc$ if for every $U'\in\Uc'$
there exists $U\in\Uc$ such that $U'\subseteq U$.

\begin{lemma}[The Scylla and Charybdis Lemma]
\label{l.Scylla_and_Charybdis}
  Let $D$ be a Jordan domain and $\Gamma$ a simple closed curve intersecting $\d D$.
  Fix two points $a,b\in\d D$
  that are not on the boundary (taken in $\d D$) of $\Gamma\cap\d D$ and
  such that, for every component $A$ of
  $\Gamma\cap D$, the endpoints of $A$ do not separate $a$ from $b$ in $\d D$.
  Then the points $a$ and $b$ can be connected
  by a simple curve lying entirely in $D$ (except for the endpoints)
  and disjoint from $\Gamma$ (except, possibly, for the endpoints).
\end{lemma}

The name of the lemma is due to the fact that parts of
$\Gamma$ can ``penetrate'' into
$D$ from ``both sides'', i.e. through both components of $\d D-\{a,b\}$.

\begin{proof}
  Let $U$ be a Jordan domain bounded by $\Gamma$ that contains $a$.
  By Proposition \ref{Jordan-intersect}, there is a component $V$ of $D\cap U$
  such that $a\in \d V$.
  (In the case $a$ is an interior point, with respect to $\d D$, of $\Gamma\cap\d D$,
  we need an obvious modification of Proposition \ref{Jordan-intersect},
  the proof is the same).
  Moreover, $V$ is a Jordan domain.
  From the construction of $V$ given in the proof of Proposition
  \ref{Jordan-intersect} it follows that $b$ is also on the boundary of $V$.
  It remains to use the Extension property (Theorem \ref{ext-prty})
  to conclude that there is a simple curve connecting $a$ and $b$ and
  lying in $V$, except for the endpoints.
  Clearly, this curve is disjoint from $\Gamma$.
\end{proof}

The following lemma is taken from \cite{Moore_postulates}
(except for the name):

\begin{lemma}[The Bump Lemma]
\label{l.bump}
  Let $D$ be a Jordan domain and $\Gamma$ a simple closed curve.
  Consider two points $a,b\in\d D$ that are not on the boundary
  (taken in $\d D$) of $\Gamma\cap \d D$.
  Then there is a simple curve $C$ with endpoints $a$ and $b$ such that
  $C-\{a,b\}\subset D$, and $C\cap\Gamma$ is finite.
\end{lemma}

\begin{proof}
  It is not hard to see that there is a simple curve $C'$ connecting $a$
  with $b$ and lying in $D$, except for the endpoints, with the following property:
  there are non-degenerate segments $ad$ and $bc$ of $C'$ that
  are disjoint from $\Gamma$, apart from (possibly) the endpoints
  (non-degenerate means containing more than one point).
  Consider the Jordan domain $U\subset D$ bounded by $C'$ and an arc of $\d D$
  between $a$ and $b$.
  These two curves will be referred to as the {\em sides} of $U$,
  and the points $a$ and $b$ as the {\em vertices} of $U$.

  Let $L$ be a component of $\Gamma\cap U$.
  We call $L$ a {\em crossing component} if the endpoints of $L$ belong to different
  sides of $U$.
  There are only finitely many crossing components (otherwise, they would accumulate
  somewhere, which contradicts the fact that $\Gamma$ is a simple closed curve).
  By the Jordan domain axiom, the crossing components subdivide $U$ into
  several Jordan subdomains $U_0$, $\dots$, $U_n$ so that $\d U_i\cap\d U_{i+1}$
  is the closure of a crossing component $L_i$ for $i=0,\dots,n-1$.
  Choose arbitrary points $a_i\in L_i$ for $i=1,\dots,n-1$, and set $a_0=a$, $a_n=b$.
  By the Scylla and Charybdis Lemma, Lemma \ref{l.Scylla_and_Charybdis},
  we can connect $a_i$ to $a_{i+1}$ by
  a simple curve lying entirely in $U_i$, except for the endpoints, and disjoint
  from $\Gamma$ (except, again, for the endpoints).

  The union of the curves thus obtained is a simple curve $C$ connecting $a$ with $b$,
  for which the intersection $C\cap\Gamma$ is finite
  (one point in every crossing component of $\Gamma$, and possibly the endpoints $a$
  and/or $b$).
\end{proof}

A corollary of the Bump Lemma is the following

\begin{cor}
\label{bump_curve}
  Let $D$ be a Jordan domain, and $\Gamma$ a simple closed curve.
  For every compact subset $K\subset D$, there exists a Jordan domain
  $\tilde D\subseteq D$ containing $K$ such that $\d\tilde D\cap\Gamma$ is finite.
\end{cor}

\begin{proof}
  Every point $x\in\d D$ has a Jordan domain neighborhood $U_x$ disjoint from $K$.
  By compactness, $\d D$ can be covered with a finite number of such neighborhoods.
  Divide $\d D$ into arcs so that every arc of this subdivision is contained in some $U_x$.  
  Consider an arc $A$ of our subdivision with endpoints $a$ and $b$.
  By construction, there is a Jordan domain $U$ containing $A$ and disjoint from $K$.
  Let $V$ be a component of $U\cap D$, whose boundary contains $A$
  (such component exists by Proposition \ref{Jordan-intersect}).
  By the Bump Lemma, Lemma \ref{l.bump}, there is a simple
  curve $C$ connecting $a$ with $b$, lying entirely in $V$, except for
  the endpoints, and intersecting $\Gamma$ in finitely many points
  (we may need to replace $a$ and $b$ with some nearby points, bounding a
  slightly larger arc, to make the Bump Lemma applicable).
  Replace $A$ with $C$ in $\d D$.
  Then we obtain a new Jordan curve.
  It is straightforward to check that the Jordan domain $D'$ bounded by the 
  new Jordan curve and contained in $D$ satisfies the following property: 
  $K\subset D'\subset D$. 
  Indeed, by Proposition \ref{div-domain}, the curve $C$ divides $D$
  into two Jordan domains $D'$ and $D''$. 
  The domain $D''$ is contained in $U$, hence is disjoint from $K$.
  It remains to repeat the same procedure with other arcs of our subdivision.
\end{proof}

If two simple closed curves have only finitely many
intersection points, then one can perturb one of the curves to
make the intersections simple:

\begin{prop}
\label{p.make_simple_intersection}
  Let $\Gamma$ be a simple closed curve, and $D$ a Jordan domain
  such that $\d D\cap \Gamma$ is finite.
  Then, for any compact subset $K\subset D$, there exists a Jordan domain
  $\tilde D$ such that $\d\tilde D$ and $\Gamma$ have only simple intersections,
  and $K\subset\tilde D\subset D$.
\end{prop}

\begin{proof}
  Indeed, let $x$ be an intersection point of $\Gamma$ and $\d D$ such that
  the two components of $\d D-\Gamma$ having $x$ as a limit point
  lie on the same side of $\Gamma$,
  i.e. in the same Jordan domain bounded by $\Gamma$.
  Let $V$ be a small neighborhood of $x$ such that $V$ is disjoint from $K$ and
  $V\cap\d D\cap\Gamma=\{x\}$.
  Consider an arc $A$ of $\d D$ lying in $V$ and containing $x$
  in its interior (taken in $\d D$).
  There is a simple curve $A'\subset V$ connecting the endpoints of $A$
  that lies in $V$, except for the endpoints,
  and has only simple intersections with $\Gamma$
  (a construction of such curve is given in the proof of the Bump Lemma,
  Lemma \ref{l.bump}).

  Let $D'$ be the Jordan domain bounded by $(\d D-A)\cup A'$
  and not containing $x$.
  Then the number of non-simple intersection points in $\d D'\cap\Gamma$ is less
  than that in $\d D\cap\Gamma$.
  Moreover, we have $K\subset D'\subset D$.
  Proceed in the same way to remove other non-simple intersections.
\end{proof}

The Covering property (Theorem \ref{cov-prty}) follows from
Proposition \ref{p.make_simple_intersection}.

\section{Quadrilaterals and grids}

As in the preceding sections,
consider a compact, connected, locally path connected, Hausdorff topological
space $X$ satisfying the Jordan domain axiom and the Basis axiom.
A {\em quadrilateral} $Q$ in $X$ is defined as a Jordan domain
with a distinguished quadruple of points $a,b,c,d\in\d Q$
(appearing on $\d Q$ in this cyclic order) called the {\em vertices} of $Q$.
The four vertices divide $\d Q$ into four (open in $\d Q$) arcs called
the {\em edges} of $Q$.
We will refer to the edges $ab$ and $cd$ as the {\em vertical edges}, and to the
edges $bc$ and $ad$ as the {\em horizontal edges}, although this terminology depends,
of course, on the labeling of vertices.

A simple curve $C$ is called {\em horizontal} with respect to a
quadrilateral $Q$ if $C\subseteq\overline Q$,
the endpoints of $C$ belong to different vertical edges, and $C$
intersects $\d Q$ only by the endpoints.
Similarly, a simple curve $C$ is called {\em vertical} with respect to $Q$ if
$C\subseteq\overline Q$, the endpoints of $C$ belong to different
horizontal edges, and $C$ intersects $\d Q$ only by the endpoints.
Define a {\em grid} in the quadrilateral $Q$ as a collection of
finitely many horizontal curves and finitely many vertical curves
such that two different horizontal curves and two different vertical curves
are disjoint, and every horizontal curve meets every vertical curve
at exactly one point.
A repeated application of Proposition \ref{div-domain} yields that a grid
consisting of $n$ horizontal and $m$ vertical curves divides $Q$ into
$(n+1)(m+1)$ Jordan domains called {\em cells} of the grid.
We will prove the existence of grids with certain properties.
To this end, we will need the following proposition, which is
a corrected version of a statement from \cite{Moore_postulates}.

\begin{prop}
\label{curve2grid}
  Let $\Gamma$ be a simple closed curve such that
  $\Gamma\cap\d Q$ is nonempty, contains simple intersections only
  (in particular, is finite), and does not contain vertices of $Q$.
  Then there exists a grid $G$ in $Q$ such that $\Gamma\cap\overline Q$ is
  contained in the union of all horizontal and vertical curves of $G$.
\end{prop}

\begin{proof}
  The intersection $\Gamma\cap\overline Q$ consists of several simple
  disjoint curves with endpoints on $\d Q$.
  Let $Q'$ be the unit square in the plane.
  There is a bijection between the finite set $P=(\Gamma\cap\d Q)\cup\{a,b,c,d\}$ and
  a finite set $P'$ of points on $\d Q'$ containing all vertices of $Q'$
  such that vertices correspond to vertices and the bijection preserves the
  cyclic order.
  On $P'$, we can define the {\em adjacency relation} as follows:
  two points are called {\em adjacent} if the corresponding points in $P$
  are connected by an arc of $\Gamma$ lying in $\overline Q$.
  It follows from Proposition \ref{crossing} that a pair of adjacent
  points cannot separate another pair of adjacent points in $\d Q'$.

  An {\em $xy$-curve} in the plane is defined as a broken line consisting
  of intervals parallel to the $x$-axis or to the $y$-axis.
  Clearly, we can connect all pairs of adjacent points in $P'$
  by disjoint $xy$-curves lying entirely in $Q'$, except for the endpoints.
  This set of disjoint $xy$-curves lies in the union of all vertical
  and horizontal intervals of some grid $G'$ in $Q'$
  consisting of vertical, i.e. parallel to the $y$-axis, straight line intervals
  and horizontal, i.e. parallel to the $x$-axis, straight line intervals.
  To construct $G'$, it suffices to extend all intervals in all $xy$-curves
  that we have.

  Using the Extension property (Theorem \ref{ext-prty})
  several times, we can extend $\Gamma\cap\overline Q$
  to a grid ``combinatorially equivalent'' (in an obvious sense) to the
  grid $G'$ in $Q'$.
\end{proof}

The following result is crucial for the proof of Theorem \ref{moore-theory}.

\begin{thm}
  \label{subordinate-grid}
  Let $\Uc$ be any open covering of $\overline Q$.
  Then there exists a grid $G$ in $Q$ that is {\em subordinate} to $\Uc$, i.e.
  such that the closure of every cell of $G$ is contained in some element of $\Uc$.
\end{thm}

\begin{proof}
  By the Basis axiom, we can assume that $\Uc$ consists of Jordan domains.
  By compactness, we can also assume that $\Uc$ is finite.
  We will now prove the following statement by induction on $n$:
  for every quadrilateral $Q$ and every covering $\Uc$ of $\overline Q$
  by $n$ Jordan domains, there exists a grid in $Q$ subordinate to $\Uc$.
  The base of induction is obvious: if $\Uc$ consists of only one
  Jordan domain $U$, then $\overline Q\subset U$, and the empty grid works.

  We can now perform the induction step.
  Let $U\in\Uc$ be a Jordan domain containing the vertex $a$ of $Q$.
  By Corollary \ref{bump_curve} and Proposition \ref{p.make_simple_intersection},
  there is a Jordan domain $V\subseteq U$ such that
  $(\Uc-\{U\})\cup\{V\}$ is still a covering of $\overline Q$, and
  $\d V$ has only simple intersections with $\d Q$.
  By Proposition \ref{curve2grid}, there is a grid $G_0$ in $Q$
  such that $\d V$ is contained in the union of all horizontal and all
  vertical curves of $G_0$.
  Let $C$ be any cell of $G_0$ not covered by $V$ (hence disjoint from $V$).
  Since $\overline C$ is covered by $\Uc-\{U\}$, we can use the
  induction hypothesis to conclude that there is a grid $G_C$ in $C$
  subordinate to $\Uc$.
  In this way, we get grids subordinate to $\Uc$ in all cells of $G_0$.
  It remains to use the Extension property (Theorem \ref{ext-prty})
  to extend all these grids to
  a single grid in $Q$, which would also be subordinate to $\Uc$.
\end{proof}

\begin{proof}[Proof of Theorem \ref{moore-theory}]
  Consider an arbitrary quadrilateral $Q$ in $X$.
  It suffices to prove that the closure of $Q$ is
  homeomorphic to the closed disk or, equivalently, to the
  standard square $[0,1]\times [0,1]$.

  By the Basis axiom, there is a countable basis $\Bc$ of the topology in $X$.
  There are countably many finite open coverings of $\overline Q$ by
  elements of $\Bc$.
  Number all such coverings by natural numbers.
  We will define a sequence of grids $G_n$ in $Q$ by induction on $n$.
  For $n=1$, we just take the empty grid, the one that does not have any
  horizontal or vertical curves.
  Suppose now that $G_n$ is defined.
  Let $\Uc_n$ be the $n$-th covering of $\overline Q$.
  Using Theorem \ref{subordinate-grid}, we can find a grid in each cell of
  $G_n$ that is subordinate to $\Uc_n$.
  Using the Extension property (Theorem \ref{ext-prty}),
  we can extend these grids to a single grid $G_{n+1}$ in $Q$.
  Thus $G_{n+1}$ contains $G_n$ and is subordinate to $\Uc_n$.

  Consider any pair of different points $x,y\in Q$.
  There exists $n$ such that $x$ and $y$ do not belong to the closure of the
  same cell in $G_{n+1}$.
  Indeed, let us first define a covering of $\overline Q$ as follows.
  For any $z\in\overline Q$, choose $U_z$ to be an element of $\Bc$
  that contains $z$ but does not include the set $\{x,y\}$.
  The sets $U_z$ form an open covering of $\overline Q$.
  Since $\overline Q$ is compact, there is a finite subcovering.
  This finite subcovering coincides with $\Uc_n$ for some $n$.
  Then, by our construction, the closure of every cell in $G_{n+1}$
  is contained in a single element of $\Uc_n$.
  However, the set $\{x,y\}$ is not contained in a single element of $\Uc_n$.
  Therefore, $\{x,y\}$ cannot belong to the closure of a single cell.

  Consider a nested sequence $C_1\subset C_2\supset\dots\supset C_n\supset\dots$,
  where $C_n$ is a cell of $G_n$.
  We claim that the intersection of the closures $\overline C_n$ is a single point.
  Indeed, this intersection is nonempty, since $\overline C_n$ form a nested
  sequence of nonempty compact sets.
  On the other hand, as we saw, there is no pair $\{x,y\}$ of different points
  contained in all $\overline C_n$.

  Consider the standard square $[0,1]\times [0,1]$ and a sequence $H_n$ of
  grids in it with the following properties:
  \begin{enumerate}
  \item all horizontal curves in $H_n$ are horizontal straight intervals, all
    vertical curves in $H_n$ are vertical straight intervals;
  \item the grid $H_n$ has the same number of horizontal curves and the same
    number of vertical curves as $G_n$, thus there is a natural one-to-one
    correspondence between cells of $H_n$ and cells of $G_n$ respecting
    ``combinatorics'', i.e. the cells in $H_n$ corresponding to adjacent
    cells in $G_n$ are also adjacent;
  \item the grid $H_{n+1}$ contains $H_n$; moreover, if a cell of $H_{n+1}$ is in
    a cell of $H_n$, then there is a similar inclusion between the corresponding
    cells of
    $G_{n+1}$ and $G_n$;
  \item between any horizontal interval of $H_n$ and the next horizontal interval,
    the horizontal intervals of $H_{n+1}$
    are equally spaced; similarly, between any vertical interval of $H_n$ and
    the next vertical interval, the vertical
    intervals of $H_{n+1}$ are equally spaced.
\end{enumerate}
It is not hard to see that any nested sequence of cells $D_n$ of $H_n$ converges
to a point:
$\bigcap\overline D_n=\{pt\}$.

We can now define a map $\Phi:\overline Q\to [0,1]\times [0,1]$ as follows.
For a point $x\in\overline Q$, there is a nested sequence of cells $C_n$
of $G_n$ such that $\overline C_n$ contains $x$ for all $n$.
Define the point $\Phi(x)$ as the intersection of the closures of the
corresponding cells $D_n$ in $H_n$.
(Note that they also form a nested sequence according to our assumptions).
Clearly, the point $\Phi(x)$ does not depend on a particular choice of the
nested sequence $C_n$ (there can be at most four different choices).
It is also easy to see that $\Phi$ is a homeomorphism between $\overline Q$ and
the standard square $[0,1]\times [0,1]$.
\end{proof}

\section{The Jordan curve theorem}

In this section, we deal with topology of the 2-sphere $S^2$.
One of the most fundamental results in topology of $S^2$ is
the {\em Jordan curve theorem}:
a simple closed curve in the 2-sphere divides the sphere into two connected components.
In this section, we prove a certain generalization of the Jordan curve theorem,
which we need to prove Theorem \ref{moore}.
This generalization is also due (mainly) to Moore.
A proof, however, can be obtained by a slight modification of a modern
standard proof of the Jordan curve theorem.

We will need the following standard facts from algebraic topology,
which we quote without proof:

\begin{thm}[Alexander's duality]
\label{alexander}
Let $U$ be an open subset of the sphere such that $S^2-U$ has $k<\infty$
connected components.
Then the first Betti number $h_1(U)$ is equal to $k-1$.
\end{thm}

\begin{thm}[Mayer--Vietoris sequence]
\label{t.MV}
Let $U_1$ and $U_2$ be open subsets of the 2-sphere.
Then the homology spaces of $U_1$, $U_2$, $U_1\cap U_2$ and
$U_1\cup U_2$ with real coefficients form the following
exact sequence
$$
0\to H_1(U_1\cap U_2)\to H_1(U_1)\oplus H_1(U_2)\to H_1(U_1\cup U_2)\map{\d_*}
$$
$$
\map{\d_*} H_0(U_1\cap U_2)\to H_0(U_1)\oplus H_0(U_2)\to H_0(U_1\cup U_2)\to 0.
$$
\end{thm}

These are partial cases of much more general theorems known by the names given
in parentheses.
Note that no fancy homology theory is needed because open sets in the
sphere are smooth manifolds.
E.g. one can use simplicial homology.
Proofs of the facts cited above (and their generalizations) can be found in standard
textbooks on algebraic topology.
See e.g. \cite{Newman} for Theorem \ref{alexander} and \cite{Munkres} for Theorem \ref{t.MV}.

Let $U_1$ and $U_2$ be open sets in $S^2$.
Then the Betti numbers of $U_1$, $U_2$, $U_1\cap U_2$ and $U_1\cup U_2$
are related as follows:
$$
h_1(U_1\cap U_2)-h_1(U_1)-h_1(U_2)+h_1(U_1\cup U_2)-
$$
$$
-h_0(U_1\cap U_2)+h_0(U_1)+h_0(U_2)
-h_0(U_1\cup U_2) = 0.
$$
This follows immediately from Theorem \ref{t.MV} and the following algebraic fact:
the alternating sum of the dimensions of vector spaces forming an exact sequence
is equal to zero.

The most nontrivial and interesting map in the Mayer--Vietoris exact sequence is
the {\em connecting homomorphism} $\d_*$.
It is defined as follows.
Represent an element of $H_1(U_1\cup U_2)$ by a simplicial cycle
$\beta=\alpha_1-\alpha_2$, where $\alpha_i$ is a simplicial chain supported in $U_i$.
The image $\d_*[\beta]$ is defined as the element of $H_0(U_1\cap U_2)$ represented
by the cycle $\d\alpha_1=\d\alpha_2$ in $U_1\cap U_2$ (note that this cycle is
homologous to zero in $U_1$ and in $U_2$ but, in general, not in $U_1\cap U_2$).
All other maps in the Mayer--Vietoris sequence are induced by the inclusions
$$
U_1\cap U_2\hookrightarrow U_1, U_2\hookrightarrow U_1\cup U_2.
$$

We say that a closed set $Z\subset S^2$ {\em separates} two points of $S^2$
if these two points belong to different connected components of $S^2-Z$.

\begin{prop}
\label{separate}
Suppose that subsets $X_1$ and $X_2$ of $S^2$ are closed and do not
separate points $a$ and $b$ in the sphere.
If $X_1\cap X_2$ is connected, then $X_1\cup X_2$ does not separate $a$ from $b$ either.
\end{prop}

\begin{proof}
Let $U_i$ denote the complement to $X_i$ in the sphere, $i=1,2$.
Thus $U_1$ and $U_2$ are open sets.
Moreover, $a$ and $b$ are in the same component of $U_1$ and in the same
component of $U_2$.
Thus, we can connect $a$ to $b$ by a simple curve $A_i$ lying entirely in $U_i$, $i=1,2$.
Moreover, we can assume that $A_i$ is the support of some simplicial chain
$\alpha_i$ oriented from $a$ to $b$.
Then $\beta=\alpha_1-\alpha_2$ is a simplicial cycle.

By definition of the boundary map (in the Mayer--Vietoris exact sequence)
$$
\d_*:H_1(U_1\cup U_2)\to H_0(U_1\cap U_2),
$$
the homology class of $\d\alpha_1=\d\alpha_2$ is equal to $\d_*([\beta])$,
where $[\beta]$ is the homology class of $\beta$.
On the other hand, $H_1(U_1\cup U_2)=0$ by Alexander's duality and
the fact that $X_1\cap X_2$ is connected.
It follows that $\d\alpha_1$ is a boundary.
But then $0=[\d\alpha_1]=[b]-[a]$, where $[a]$ and $[b]$ are classes of points
$a$ and $b$ in $H_0(U_1\cap U_2)$.
It follows that $a$ and $b$ are in the same connected component of $U_1\cap U_2$,
i.e. are not separated by $X_1\cup X_2$.
\end{proof}

Recall that a map $F:t\mapsto F(t)$ assigning a compact subset $F(t)$ in
a Hausdorff space $X$ to every element $t$ of some topological space $T$
is called {\em upper semicontinuous} if for every $t\in T$
and every open neighborhood $U$ of $F(t)$, there is an open neighborhood $V$
of $t$ with the property $F(t')\subset U$ for all $t'\in V$.

\begin{prop}
\label{usc}
  Consider a map $F$ of the interval $[0,1]$ into the set of compact subsets in $S^2$
  such that $F(t)\cap F(t')=\emptyset$ for $t\ne t'$.
  Then $F$ is upper semi-continuous if and only if for every closed subinterval $A$
  of $[0,1]$, the union
  $$
  F(A)=\bigcup_{t\in A} F(t)
  $$
  is closed.
  The same statement is true if, instead of just subintervals, we consider all
  closed subsets in $[0,1]$.
\end{prop}

A similar statement holds, in which $[0,1]$ is replaced with $S^1$, with the same proof.

\begin{proof}
  Suppose that $F(A)$ is closed for every closed subinterval $A\subset [0,1]$.
  Choose a point $t_\infty$ in $[0,1]$ and a
  nested sequence $A_n$ of subintervals in $[0,1]$ containing $t_\infty$
  in their relative interiors (taken in $[0,1]$), whose intersection is $\{t_\infty\}$.
  Clearly, the intersection of compact sets $F(A_n)$ is then $F(t_\infty)$
  (this follows from the fact that $F(t)$ is disjoint from $F(t')$ for $t\ne t'$).
  Let $U$ be any open neighborhood of $F(t_\infty)$ in the sphere.
  The sets $F(A_n)-U$ form a nested sequence of compact sets, whose intersection
  is empty.
  Therefore, all these sets must be empty for $n$ big enough, i.e. $F(t)\subset U$
  for all $t\in A_n$.

  Conversely, suppose that the map $F$ is upper semicontinuous.
  Let $A$ be a closed subset of $[0,1]$ (not necessarily a subinterval).
  We need to prove that the set $F(A)$ is closed.
  Indeed, choose any sequence $x_n\in F(t_n)$ that converges to some point $x_\infty$ of the sphere.
  Passing to a subsequence if necessary, we can assume that the sequence $t_n$ converges
  to some $t_\infty\in A$.
  Since the sets $F(t_n)$ accumulate on $F(t_\infty)$, we must have
  $x_\infty\in F(t_\infty)\subset F(A)$.
\end{proof}

\begin{prop}
  \label{usc_conn}
  Consider an upper semicontinuous map $F$
  from $[0,1]$ to the set of compact subsets in $S^2$.
  If $F(t)$ is connected for every $t\in [0,1]$,
  then the union $F[0,1]$ is also connected.
\end{prop}

\begin{proof}
  Assume the contrary: $F[0,1]$ splits into a disjoint union of two 
  closed subsets. 
  Since all $F(t)$, $t\in [0,1]$, are connected, every term of 
  this splitting must be a union of sets $F(t)$, i.e.
  the splitting must have the form $F(A)\sqcup F(B)$,
  where $[0,1]=A\sqcup B$, and both $F(A)$ and $F(B)$ are closed.
  However, since $[0,1]$ is connected, we cannot have that both $A$ and $B$
  are closed.
  Suppose, say, that $\overline A\cap B\ne\emptyset$.
  Then, for every $t\in \overline A\cap B$, we must have
  $F(t)\subset \overline{F(A)}\cap F(B)$,
  which contradicts the fact that $F(A)$ and $F(B)$ are disjoint.
\end{proof}

A subset of $S^2$ is called {\em non-separating} if it does not separate the sphere.
The following theorem generalizes the fact that a simple curve is non-separating
\cite{Jan}.

\begin{thm}
  \label{non-sep}
  Consider an upper semicontinuous map $F$ from $[0,1]$
  to the set of compact connected subsets of the sphere.
  If all $F(t)$ are disjoint and non-separating, then $F[0,1]$ is non-separating.
\end{thm}

\begin{proof}
Assume the contrary: there are two points $a$ and $b$ in the complement to
$F[0,1]$ such that $F[0,1]$ separates $a$ from $b$.
Consider two compact sets $F[0,1/2]$ and $F[1/2,1]$.
Their intersection $F(1/2)$ is connected.
By Proposition \ref{separate}, either $F[0,1/2]$ or $F[1/2,1]$ separates $a$ from $b$.
We can continue the same process to obtain a nested sequence of subintervals
$A_n\subset [0,1]$ such that $F(A_n)$ separate $a$ from $b$ and
$\bigcap A_n$ is a single point $t\in [0,1]$.
However, we know that $F(t)$ is non-separating.
Choose a simple curve $C$ connecting $a$ to $b$ in the complement to $F(t)$.
There is a neighborhood $U$ of $F(t)$ disjoint from $C$.
The sets $F(A_n)$ must be contained in $U$ for all large $n$.
A contradiction with the fact that $F(A_n)$ separates $a$ from $b$.
\end{proof}

The main theorem of this section, which is a generalization of the
Jordan curve theorem, is the following:

\begin{thm}
\label{Jordan}
Consider an upper semicontinuous map $F$ from the circle $S^1$
to the set of compact connected subsets in the sphere.
Suppose that all $F(t)$ are disjoint, and do not separate the sphere.
Then the complement to the set $F(S^1)$ consists of exactly two connected components.
\end{thm}

\begin{proof}
Consider two closed arcs $A_1$ and $A_2$ of $S^1$ such that $A_1\cap A_2$ is a pair of points
(the common endpoints of $A_1$ and $A_2$) and $S^1 = A_1\cup A_2$ (the second condition follows
from the first).
Set $X_i = F(A_i)$ and $U_i = S^2-X_i$.
We can use the Mayer--Vietoris theorem \ref{t.MV} for $U_1$ and $U_2$.
What we want to know is the term $h_0(U_1\cap U_2)$.
It turns out that we can compute all other terms.
By Theorem \ref{alexander} and Proposition \ref{usc_conn}, we know all $h_1$-terms:
$$
h_1(U_1\cap U_2)=h_1(U_1)=h_1(U_2)=0,\quad h_1(U_1\cup U_2)=1.
$$
By Proposition \ref{non-sep}, the set $X_i$ does not separate the sphere, hence $h_0(U_i) = 1$.
The complement to $U_1\cup U_2$ is the union of two disjoint compact connected non-separating sets.
By Proposition \ref{separate}, this union does not separate the sphere.
Therefore, $h_0(U_1\cup U_2) = 1$.
By Theorem \ref{t.MV}, we can now conclude that $h_0(U_1\cap U_2)=2$ as stated.
\end{proof}

\section{Quotients of the sphere}

In this section, we prove Theorem \ref{moore}.
Consider a closed equivalence relation $\sim$ on $S^2$ satisfying the assumptions of
the theorem.
The quotient space of a compact connected Hausdorff space by a closed equivalence
relation is also compact, connected and Hausdorff.
In addition, if the space is locally path connected, and all equivalence classes
are connected, then the quotient is also locally path connected.
See e.g. \cite{Kuratowski}.
Thus we know that $X = S^2/\sim$ is compact, Hausdorff, connected and locally
path connected.
To prove Theorem \ref{moore}, we need to show that $X$ satisfies the
Jordan domain axiom and the Basis axiom.
Let $\pi:S^2\to X$ denote the canonical projection.

\begin{thm}
\label{Jordan-in-X}
The space $X$ satisfies the Jordan domain axiom.
Namely, for every simple closed curve $C$ in $X$,
the complement to $C$ consists of exactly two connected components.
The boundary of each of these components coincides with $C$.
\end{thm}

As before, we call these connected components the
{\em Jordan domains} bounded by the curve $C$.

\begin{proof}
Consider a closed simple path $\gamma:S^1\to X$ parameterizing $C$,
i.e. such that $\gamma(S^1) = C$.
Define the function $F$ from $S^1$ to the set of compact subsets in the sphere by
the formula $F(t)=\pi^{-1}(\gamma(t))$.
The sets $F(t)$ are disjoint, connected and non-separating.
For any closed (hence compact) arc $A\subseteq S^1$, the set $\gamma(A)$ is compact
(hence closed) in $X$, therefore, the set $\pi^{-1}(\gamma(A))$ is closed in $S^2$.
It follows that $F$ satisfies the assumptions of Theorem \ref{Jordan}.
By this theorem, the complement to $F(S^1)=\pi^{-1}(C)$ splits into two disjoint
open connected sets.
These sets project to some open connected sets $U_1$ and $U_2$ in $X$.
Clearly, $X = C\sqcup U_1\sqcup U_2$.

It remains to prove that $\d U_i=C$ for $i=1,2$.
Clearly, $\d U_i\subseteq C$.
We want to show that the opposite inclusion holds.
Take a point $x=\gamma(t)\in C$ and a neighborhood $V$ of $x$.
It suffices to prove that $V$ will necessarily intersect $\d U_1$ and $\d U_2$.
Let $I$ be a closed arc of $S^1$ such that $t\not\in I$ and $I\cup\gamma^{-1}(V)=S^1$.
By Theorem \ref{non-sep}, the set $\gamma(I)$ does not separate $X$.
It follows that $x$ can be connected to some point in $U_1$ and
to some point in $U_2$ by continuous paths avoiding $\gamma(I)$.
Since $x$ is neither in $U_1$ nor in $U_2$, these paths must
intersect $\d U_1$ and $\d U_2$ (respectively),
and these intersections can only happen in $V$.
\end{proof}

It remains to prove that $X$ satisfies the Basis axiom.

\begin{thm}
\label{Jordanization}
  Let $x,y\in X$ be two points, and $\alpha:S^1\to S^2$ a simple closed path
  that separates $\pi^{-1}(x)$ from $\pi^{-1}(y)$.
  Then there is a simple closed curve $\Gamma$ in $X$ that separates
  $x$ from $y$ and such that $\Gamma\subseteq\pi\circ\alpha(S^1)$.
\end{thm}

In the proof given below, we assume that $\alpha$ is a smooth embedding.
This is sufficient for our purposes.
The proof can be easily extended to the general case using relative homology.

\begin{proof}
  Let $A$ be an open arc of $S^1$.
  Suppose that the endpoints of $A$ belong to $\pi^{-1}(C)$
  for some equivalence class $C$ of $\sim$.
  Consider a smooth curve $\beta$ connecting $\pi^{-1}(x)$ to
  $\pi^{-1}(y)$ in $S^2$ and avoiding $C$
  (i.e. the endpoints of $\beta$ belong $\pi^{-1}(x)$ and $\pi^{-1}(y)$,
  respectively, and $\beta\cap C=\emptyset$).
  By a small perturbation, we can make $\beta$ transverse to $\alpha(A)$.
  We say that $A$ is {\em even} (respectively, {\em odd}), if every such $\beta$
  intersects $\alpha(A)$ even (respectively, odd) number of times.
  Note that the parity does not depend on $\beta$ provided that
  $\beta$ satisfies our assumptions, i.e. connects $\pi^{-1}(x)$ to
  $\pi^{-1}(y)$, avoids $C$ and is transverse to $\alpha(A)$.
  Indeed, consider two such curves $\beta_1$ and $\beta_2$.
  Let $\beta_i\cdot\alpha(A)\in\Z/2\Z$ be the residue modulo 2
  that represents the parity of the cardinality of $\beta_i\cap\alpha(A)$,
  $i=1,2$.
  The sum $\gamma=\beta_1+\beta_2$ represents a cycle in $H_1(S^2-C,\Z/2\Z)$.
  Since $H_1(S^2-C,\Z/2\Z)=0$, this cycle is homologous to 0.
  On the other hand, $\beta_1\cdot\alpha(A)+\beta_2\cdot\alpha(A)$
  is the image of $([\gamma],[\alpha(A)])$ under the Poincar\'e pairing
  $$
  H_1(S^2-C,\Z/2\Z)\times H_1^c(S^2-C,\Z/2\Z)\to\Z/2\Z.
  $$
  Here $H_1^c$ is the first homology space with compact support; the 
  curve $\alpha(A)$ defines a homology class $[\alpha(A)]$ in $H_1^c(S^2-C,\Z/2\Z)$.
  This image is zero, therefore, $\beta_1\cdot\alpha(A)=\beta_2\cdot\alpha(A)$.
  The proof will now consist of several steps.

  {\em Step 1.}
  Let $A_n$ be a sequence of even arcs that converges to some arc $A$.
  We will now prove that $A$ is also an even arc.
  Suppose that the endpoints of $A_n$ belong to $\alpha^{-1}(C_n)$
  for an equivalence class $C_n$ of $\sim$.
  Clearly, $C_n$ accumulate on some equivalence class $C$,
  and the endpoints of $A$ belong to $\alpha^{-1}(C)$.
  Consider a smooth curve $\beta$ connecting $\pi^{-1}(x)$ to
  $\pi^{-1}(y)$, avoiding $C$, and transverse to $\alpha(A)$.
  Since $C_n$ accumulate on $C$, the path $\beta$ is disjoint from
  $C_n$ for sufficiently large $n$.
  It follows that $\beta\cdot\alpha(A_n)=0\pmod 2$.
  Moreover, $\beta\cap\alpha(A_n)=\beta\cap\alpha(A)$ for
  sufficiently large $n$.
  Therefore, $A$ is even.

  {\em Step 2.}
  Let $A_1$ be a longest even arc of $S^1$.
  From Step 1, it follows that such arc exists.
  (We measure the lengths with respect to some fixed metric on $S^1$).
  Denote by $C_1$ the equivalence class of $\sim$ such that
  $\alpha^{-1}(C_1)$ contains the endpoints of $A_1$.
  By induction, we define $A_n$ as a longest even arc essentially disjoint
  from $A_1,\dots,A_{n-1}$.
  We set $C_n$ to be the equivalence class of $\sim$ such that $\alpha^{-1}(C_n)$
  contains the endpoints of $A_n$.
  Let $K_n$ be the complement to the union $A_1\cup\dots\cup A_n$,
  and $K$ the intersection of all $K_n$.
  Set $\Gamma_n=\pi\circ\alpha(K_n)$ and $\Gamma=\pi\circ\alpha(K)$.
  The set
  $$
  \alpha(K_n)\cup\bigcup_{i=1}^n C_i
  $$
  separates $\pi^{-1}(x)$ from $\pi^{-1}(y)$.
  Indeed, let $\beta$ be a smooth curve connecting $\pi^{-1}(x)$ to
  $\pi^{-1}(y)$, avoiding $C_1,\dots,C_n$, and transverse to $\alpha(S^1)$.
  There is a component $I$ of $K_n$ such that $\beta$ intersects $\alpha(I)$
  an odd number of times.
  Indeed, $\beta$ must intersect $\alpha(S^1)$ in an odd number of points,
  but it intersects the set $\alpha(A_1\cup\dots\cup A_n)$ even
  number of times.
  It follows that $\beta\cap\alpha(K_n)$ is nonempty.
  Therefore, $\Gamma_n$ separates $x$ from $y$.

  {\em Step 3.}
  We now prove that $\Gamma$ separates $x$ from $y$.
  Suppose not.
  Then $x$ and $y$ lie in the same component of the complement to $\Gamma$.
  Since $\Gamma$ is compact, this component must be open.
  By local path connectivity, there is a continuous path $\gamma:[0,1]\to X$
  such that $\gamma(0)=x$, $\gamma(1)=y$, and $\gamma[0,1]$ avoids $\Gamma$.
  However, for every $n$, the set $\Gamma_n$ separates $x$ from $y$.
  It follows that $\gamma(t_n)\in\Gamma_n$ for some $t_n\in [0,1]$.
  Let $t$ be any limit point of the sequence $t_n$.
  Then $\gamma(t)\in\Gamma$, a contradiction.

  {\em Step 4.}
  It remains to prove that $\Gamma$ is a simple closed curve in $X$.
  In other words: if $\pi\circ\alpha(s)=\pi\circ\alpha(t)$ for
  two different points $s$ and $t$ in $K$, then
  $s$ and $t$ are the endpoints of some $A_n$.
  In any case, $s$ and $t$ belong to $\alpha^{-1}(C)$ for some equivalence
  class $C$ of $\sim$.
  One of the arcs bounded by $s$ and $t$ must be even --- if both 
  arcs were odd, then a simple curve connecting $x$ with $y$ and
  intersecting $\alpha(S^1)$ transversally would have an even number of 
  intersection points with $\alpha(S^1)$, a contradiction.
  Let $A$ be the even arc with endpoints $s$ and $t$.
  Note that, since $s$ and $t$ belong to $K$, every $A_n$ is 
  either contained in $A$ or is essentially disjoint from $A$.
  Set $n$ to be the smallest positive integer such that $A_n\subseteq A$.
  If there is no such integer, then $A$ is essentially disjoint from all $A_n$.
  On the other hand, the length of $A_n$ must tend to zero.
  Therefore, there will be some $m$ for which $A_m$ is shorter than $A$.
  A contradiction with the choice of $A_m$, which shows that $n$ is well defined.
  Now, by the choice of $A_n$, we must have $A_n=A$.
  It follows that $s$ and $t$ are the endpoints of $A_n$.
\end{proof}

We can now prove the Basis axiom for $X$, thus completing the proof of Theorem \ref{moore}.
Given a point $x\in X$ and an open set $U$ containing $x$, we need to find
a Jordan domain in $X$ containing $x$ and contained in $U$.
Let $D$ be a small Jordan neighborhood of $\pi^{-1}(x)$ compactly contained in
$\pi^{-1}(U)$, and $E$ a smaller Jordan neighborhood of $\pi^{-1}(x)$
such that every equivalence class of $\sim$ intersecting $\overline E$
is contained in $D$ (the existence of $E$ follows from the fact
that the equivalence relation $\sim$ is closed).
The boundary of $E$ is a simple closed curve $C$ such that $\pi(C)$
lies in $\pi(D)\subseteq U$.
Choose any point $y\in X-U$.
Then $C$ separates $\pi^{-1}(x)$ from $\pi^{-1}(y)$.
By Theorem \ref{Jordanization}, there exists a simple closed curve
$\Gamma\subseteq\pi(C)\subset U$ that separates $x$ from $y$.
We claim that the Jordan domain $V$ bounded by $\Gamma$ and containing
$x$ is a subset of $U$.
Indeed, the boundary of $V$ is disjoint from the connected
set $\pi(S^2-\overline D)$ (if $\pi(C)$ intersects $\pi(S^2-\overline D)$, then
$C$ intersects some equivalence class of $\sim$ not lying in $D$, a contradiction).
Since $y$ is in $\pi(S^2-\overline D)$ but not in $V$,
the set $V$ must also be disjoint from $\pi(S^2-\overline D)\supseteq X-U$.
This concludes the proof of the Basis axiom.

\section{A relative version of Theorem \ref{moore-theory}}
\label{s_rel}

In this section, we extend Theorem \ref{moore-theory} to make it applicable in a wider
variety of contexts.
Consider a compact connected Hausdorff space $X$.
Let $\Ec$ be a set of simple curves in $X$.
Suppose that any simple curve that is a subset of a countable union of curves in $\Ec$
is also an element of $\Ec$.
In particular, any segment of a simple curve in $\Ec$ is also a simple curve in $\Ec$.
Define the set $\Ec^{\circ}$ of simple closed curves as follows: a simple
closed curve $C$ belongs to $\Ec^\circ$ if all proper arcs of $C$ are in $\Ec$.
If the set $\Ec$ is fixed, we will refer to elements of $\Ec$ as {\em elementary curves},
and to elements of $\Ec^\circ$ as {\em elementary closed curves}.

The Jordan domain axiom and the Basis axiom have the following relative versions
with respect to $\Ec$:
\begin{enumerate}
  \item {\em Relative Jordan domain axiom.}
  Every elementary closed curve $C$ divides $X$ into two connected components
  such that the boundary of each component is equal to $C$.
  The Jordan domains bounded by $C$ are called {\em elementary domains}.
  \item {\em Relative Basis axiom.}
  There is a countable basis of topology in $X$ consisting of elementary domains.
\end{enumerate}

\begin{thm}
  \label{rel-moore}
  Let a space $X$ and a set $\Ec$ of simple curves in $X$ be as above.
  Suppose that, for every connected open set $U\subseteq X$, every pair of points in $U$
  can be connected by an elementary curve lying in $U$.
  Suppose also, that $X$ satisfies the relative Jordan domain axiom, and the
  relative Basis axiom with respect to $\Ec$.
  Then $X$ is homeomorphic to the 2-sphere.
\end{thm}

The proof of Theorem \ref{rel-moore} is the same as that
of Theorem \ref{moore-theory},
with simple curves and simple closed curves replaced with elementary curves
and elementary closed curves, respectively.
The fact that makes everything work is the following.
Every time we formed a new simple curve, or a new simple closed curve,
we either stayed in a countable union of existing simple curves, or
used the path connectivity of a connected open set.


\begin{thebibliography}{9999}

\bibitem[B]{Bing}
R.H. Bing, ``The Kline sphere characterization problem'',
Bull. Amer. Math. Soc. \textbf{52} (1946), 644--653

\bibitem[J]{Jan}
S. Janiszewski, ``Sur les coupures du plan faites par les continus'',
Prace Matematyczno-Fizyczne, \textbf{26} (1913)

\bibitem[vK]{vK}
E. R. van Kampen, ``On some characterizations of 2-dimensional manifolds'',
Duke Math. J. vol. \textbf{1} (1935) pp. 74--93

\bibitem[K]{Kuratowski}
K. Kuratowski, ``Topology'', Academic Pr / PWN; Revised edition, 1966

\bibitem[M16]{Moore_foundations}
R.L. Moore, ``On the foundations of plane analysis situs'',
Transactions of the AMS, \textbf{17} (1916), 131--164

\bibitem[M19]{Moore_postulates}
R.L. Moore, ``Concerning a set of postulates for plane analysis situs'',
Transactions of the AMS, \textbf{20} (1919), 169--178

\bibitem[M20]{Moore_simple}
R.L. Moore, ``Concerning simple continuous curves'',
Transactions of the AMS, \textbf{21} (1920), 333--347

\bibitem[M24]{Moore_separate}
R.L. Moore, ``Concerning the prime parts of certain continua which separate the plane'',
Proceedings of the National Academy of Sciences, \textbf{10} (1924), 170--175

\bibitem[M25]{Moore}
R.L. Moore, ``Concerning upper-semicontinuous collections of continua'',
Transactions of the AMS, \textbf{27}, Vol. 4 (1925), 416--428

\bibitem[M25Z]{Moore_prime}
R.L. Moore, ``Concerning the prime parts of a continuum'',
Mathematische Zeitschrift, \textbf{22} (1925), 307--315

\bibitem[Mu]{Mullikin}
A. Mullikin, ``Certain theorems relating to plane connected point sets'',
Transactions of the AMS, \textbf{24} (1922), 144--162

\bibitem[Mun]{Munkres}
J. Munkres, ``Topology'' (Second Ed.) Prentice Hall, 2000

\bibitem[N]{Newman}
M.H.A. Newman, ``Elements of the Topology of Plane Sets of Points'',
Greenwood Press Reprint; 2nd Ed, 1985

\bibitem[T]{Tim}
V. Timorin, ``Topological regluing of rational functions'', Invent. Math. (2009)
DOI: 10.1007/s00222-009-0220-8

\bibitem[Z]{Zippin}
L. Zippin, ``A study of continuous curves and their relation to the Janiszewski--Mullikin
Theorem'', Trans. Amer. Math. Soc. vol. \textbf{31} (1929) pp. 744--770

\end{thebibliography}
\end{document}